\documentclass[11pt,reqno]{article}

\usepackage{amsfonts,amsmath,amssymb,amsthm,latexsym,layout,graphicx,subfigure,ytableau}
\usepackage{kotex}
\usepackage{tikz}
\usepackage{tkz-euclide}
\usepackage{rotating}
\usetikzlibrary{angles,quotes}

\newtheorem{thm}{Theorem}[section]
\newtheorem{cor}[thm]{Corollary}
\newtheorem{conj}[thm]{Conjecture}
\newtheorem{lem}[thm]{Lemma}

\newtheorem{prop}[thm]{Proposition}

\newtheorem{rem}[thm]{Remark}
\newtheorem{example}[thm]{Example}
\theoremstyle{remark}
\newcommand{\la}{\lambda}

\def\la{\lambda}
\definecolor{dredcolor}{rgb}{0.9,0.3,0.4}

\title{Counting self-conjugate $(s,s+1,s+2)$-core partitions}

\author{Hyunsoo Cho\\
\small Department of Mathematics\\[-0.8ex]
\small Yonsei University\\[-0.8ex] 
\small Seoul, Republic of Korea\\
\small\tt coconut@yonsei.ac.kr\\
\and
JiSun Huh\\  
\small Applied Algebra and Optimization Research Center\\[-0.8ex]
\small Sungkyunkwan University\\[-0.8ex]
\small Suwon, Republic of Korea\\
\small\tt hyunyjia@g.skku.edu\\
\and
Jaebum Sohn\thanks{Supported by the National Research Foundation of Korea (NRF) NRF-2017R1A2B4009501.}\\
\small Department of Mathematics\\[-0.8ex]
\small Yonsei University\\[-0.8ex]
\small Seoul, Republic of Korea\\
\small\tt jsohn@yonsei.ac.kr}

\begin{document}

\maketitle

\begin{abstract}
We are concerned with counting self-conjugate $(s,s+1,s+2)$-core partitions. A Motzkin path of length $n$ is a path from $(0,0)$ to $(n,0)$ which stays above the $x$-axis and consists of the up $U=(1,1)$, down $D=(1,-1)$, and flat $F=(1,0)$ steps. We say that a Motzkin path of length $n$ is symmetric if its reflection about the line $x=n/2$ is itself. In this paper, we show that the number of self-conjugate $(s,s+1,s+2)$-cores is equal to the number of symmetric Motzkin paths of length $s$, and give a closed formula for this number. 
\end{abstract}

\section{Introduction}

Let $\la=(\la_1,\la_2,\dots,\la_{\ell})$ be a partition of a positive integer $n$. The {\it Young diagram} of $\la$ is a collection of $n$ boxes in $\ell$ rows with $\la_i$ boxes in row $i$. For example, the Young diagram for $\la=(5,4,2)$ is below.

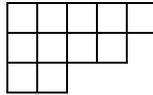
\begin{center}
\tiny{
\begin{ytableau}
~&~&~&~&~ \\
~&~&~&~ \\
~&~ 
\end{ytableau}}
\end{center}

Let the leftmost column be column 1. The box in row $i$ and column $j$ is said to be in position $(i,j)$. For the Young diagram of $\la$, the partition $\la '=(\la_1 ',\la_2 ',\dots, \la_{\la_1} ')$ is called the {\it conjugate} of $\la$, where $\la'_j$ denotes the number of boxes in column $j$. A partition whose conjugate is equal to itself is called {\it self-conjugate}.
For each box in its Young diagram, we define its {\it hook length} by counting the number of boxes directly to its right or below, including the box itself. Equivalently, for the box in position $(i,j)$, the hook length of $\la$ is defined by 
$$h(i,j)=\la_i+\la'_j-i-j+1.$$  
For a positive integer $t$, a partition $\la$ is called a {\it $t$-core} if none of its hook lengths are multiples of $t$. We use the notation of a $(t_1,...,t_p)$-core if it is simultaneously a $t_1$-core,\dots, and a $t_p$-core. See for details \cite{AL,Anderson,AHJ,GKS,JK,Johnson,Wang}.\\

For a set $S$ of positive integers, we say that $a$ is generated by $S$ if $a$ can be written as a non-negative linear combination of the elements of $S$. Let $P=P_S$ be the set of elements which are not generated by $S$, and let $(P,<_P)$ be a poset by defining the cover relation so that $a$ covers $b$ if and only if $a-b\in S$. For example, see Figure \ref{fig:p8} for the poset $P_{\{8,9,10\}}$. For the detailed explanation of poset, we refer the reader to \cite{AL,S2,YZZ}.

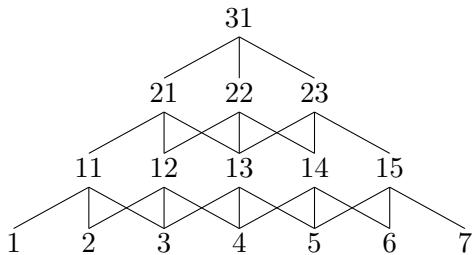
\begin{figure}[ht!]
\centering
\begin{tikzpicture}[scale=1]

\foreach \i in {1,2,3,4,5,6,7}
{\path (\i,0) coordinate (\i);
\node[above] at (\i) {\i};}

\foreach \i in {11,12,13,14,15}
{\path (-9+\i,1) coordinate (\i);
\node[above] at (\i) {\i};
\draw (\i) -- (-10+\i,0.45);
\draw (\i) -- (-9+\i,0.45);
\draw (\i) -- (-8+\i,0.45);
}

\foreach \i in {21,22,23}
{\path (-18+\i,2) coordinate (\i);
\node[above] at (\i) {\i};
\draw (\i) -- (-19+\i,1.45);
\draw (\i) -- (-18+\i,1.45);
\draw (\i) -- (-17+\i,1.45);
}

\path (4,3) coordinate (31);
\node[above] at (31) {31};
\draw (31) -- (3,2.45);
\draw (31) -- (4,2.45);
\draw (31) -- (5,2.45);
\end{tikzpicture}
\caption{The Hasse diagram of $P_{\{8,9,10\}}$}\label{fig:p8}
\end{figure}

For a poset $(P,<_P)$, a set $I\subset P$ is called a \emph{lower ideal} of $P$ if $a<_P b$ and $b\in I$ implies $a\in I$.
In \cite{Anderson}, Anderson gave a natural bijection between $t$-cores and lower ideals of a poset $P_{\{t\}}$. Moreover, she proved that for relatively prime positive integers $s$ and $t$, the number of $(s,t)$-cores has a nice closed formula by finding a bijection between $(s,t)$-cores and lattice paths from $(0,0)$ to $(s,t)$ consisting of north and east steps which stay above the diagonal.

\begin{thm}\cite{Anderson}\label{thm:anderson}
For relatively prime positive integers $s$ and $t$, the number of $(s,t)$-cores is
$$\frac{1}{s+t}\binom{s+t}{s}.$$ 
\end{thm}

Since the work of Anderson, the topic counting simultaneous cores has received growing attention. In \cite{FMS}, Ford, Mai, and Sze proved the following analog of Anderson's work.

\begin{thm}\cite{FMS}\label{thm:FMS}
For relatively prime positive integers $s$ and $t$, the number of self-conjugate $(s,t)$-cores is
$$\binom{\lfloor \frac{s}{2}\rfloor+\lfloor \frac{t}{2}\rfloor}{\lfloor \frac{s}{2}\rfloor}.$$ 
\end{thm}

An \emph{$(s,k)$-generalized Dyck path} is a path from $(0,0)$ to $(s,s)$ which stays above the diagonal and consists of the steps $N_k=(0,k)$, $E_k=(k,0)$, and $D_i=(i,i)$ for $1\leq i \leq k-1$. For example, an $(s,1)$-generalized Dyck path is a (classical) \emph{Dyck path of order} $s$. We say that an $(s,k)$-generalized Dyck path is \emph{symmetric} if its reflection about the line $y=s-x$ is itself. It is often observed that counting the number of simultaneous cores can sometimes be described as counting the number of different paths.

\begin{rem} \label{rem:symmetric} Let $s$ be a positive integer.
\begin{enumerate}
\item[1.] The number of $(s,s+1)$-cores is the $s$th Catalan number $C_s=\frac{1}{s+1}\binom{2s}{s}$ which counts the number of Dyck paths of order $s$.
\item[2.] The number of self-conjugate $(s,s+1)$-cores is $\binom{s}{\lfloor s/2 \rfloor}$ which counts the number of \emph{symmetric Dyck paths} of order $s$.
\end{enumerate}
\end{rem}

In \cite{AL}, Amdeberhan and Leven expand Anderson's result to $(s,s+1,\dots,s+k)$-cores.

\begin{thm}\cite{AL}\label{thm:AL} 
The followings are equinumerous:
\begin{enumerate}
\item[(a)] The number of $(s,s+1,\dots,s+k)$-cores.
\item[(b)] The number of $(s,k)$-generalized Dyck paths.
\item[(c)] The number of lower ideals in $P_{\{s,s+1,\dots,s+k\}}$.
\end{enumerate} 
\end{thm}

We note that $(s,2)$-generalized Dyck paths are equivalent to Motzkin paths of length $s$. From Theorem \ref{thm:AL}, one can obtain the following corollary.

\begin{cor} \label{cor:Motzkin}
For a positive integer $s$, the number of $(s,s+1,s+2)$-cores is 
$$M_s=\sum_{i=0}\frac{1}{i+1}\binom{s}{2i} \binom{2i}{i},$$
the $s$th Motzkin number which counts the number of Motzkin paths of length $s$.
\end{cor}

We note that Yang, Zhong, and Zhou \cite{YZZ} proved Corollary \ref{cor:Motzkin} independently.\\

It is natural to ask whether the number of self-conjugate $(s,s+1,s+2)$-cores and the number of symmetric Motzkin paths of length $s$ are equinumerous from Remark \ref{rem:symmetric} and Corollary \ref{cor:Motzkin}. In this paper, we prove that these two quantities are equal by showing that they satisfy the same recurrence relation which will be proved in Seciton \ref{sec:counting}. Furthermore, we give a closed formula for these numbers.


\section{Poset structure for self-conjugate $(s,s+1,s+2)$-cores}\label{sec:poset}

In this section, we construct a poset whose lower ideals with some restrictions are corresponding to self-conjugate $(s,s+1,s+2)$-cores, and then give a simple diagram to visualize that poset.\\

For a partition $\la$, let $MD(\la)$ denote the set of main diagonal hook lengths. Therefore, $MD(\la)$ is a set of distinct odds when $\la$ is self-conjugate. In \cite{FMS}, authors gave a useful result for determining self-conjugate $t$-cores. 

\begin{prop}\cite{FMS}\label{prop:FMS}
Let $\la$ be a self-conjugate partition. Then $\la$ is a $t$-core partition if and only if both of the following hold:
\begin{enumerate}
\item[(a)] If $h\in MD(\la)$ with $h>2t$, then $h-2t\in MD(\la)$.
\item[(b)] If $h_1,h_2\in MD(\la)$, then $h_1+h_2\not\equiv 0 \pmod{2t}$. 
\end{enumerate}
\end{prop}

For a positive integer $s$, we consider an induced subposet of $P=P_{\{2s,2s+1,\dots,2s+4\}}$,
$$\tilde{P}_{\{s,s+1,s+2\}}=\{h\in P~:~s\not<_P h,~s+1\not<_P h,~s+2\not<_P h, \text{~and~} h \text{~is odd} \}.$$
We note that the poset $\tilde{P}_{\{s,s+1,s+2\}}$ is the disjoint union of two posets, say $Q$ and $R$, where $Q$ is the maximal induced subposet of $P$ of which minimal elements are odd integers less than $s$, and $R$ is the maximal induced subposet of $P$ of which minimal elements are odd integers $x$ such that $s+2<x<2s$. See Figures \ref{fig:tilde8} and \ref{fig:tilde9} for example.

\begin{figure}[ht!]
\centering
\begin{tikzpicture}[scale=0.77]

\foreach \i in {1,3,5,7,11,13,15}
{\path (\i,0) coordinate (\i);
\node[above] at (\i) {\i};}

\foreach \i in {2,4,6,8,9,10,12,14}
{\path (\i,0) coordinate (\i);
\node[above, gray!60] at (\i) {\i};}

\foreach \i in {22,24,25,26,27,28,29,30}
{\path (-18+\i,1.5) coordinate (\i);
\node[above, gray!60] at (\i) {\i};
\draw[gray!60] (\i) -- (-16+\i,0.55);
\draw[gray!60] (\i) -- (-17+\i,0.55);
\draw[gray!60] (\i) -- (-18+\i,0.55);
\draw[gray!60] (\i) -- (-19+\i,0.55);
\draw[gray!60] (\i) -- (-20+\i,0.55);
}

\foreach \i in {21,23,31}
{\path (-18+\i,1.5) coordinate (\i);
\node[above] at (\i) {\i};
\draw[thick] (\i) -- (-16+\i,0.55);
\draw[gray!60] (\i) -- (-17+\i,0.55);
\draw[thick] (\i) -- (-18+\i,0.55);
\draw[gray!60] (\i) -- (-19+\i,0.55);
\draw (\i) -- (-20+\i,0.55);
}

\foreach \i in {41,42,43,44,45,46,47}
{\path (-36+\i,3) coordinate (\i);
\node[above, gray!60] at (\i) {\i};
\draw[gray!60] (\i) -- (-34+\i,2.05);
\draw[gray!60] (\i) -- (-35+\i,2.05);
\draw[gray!60] (\i) -- (-36+\i,2.05);
\draw[gray!60] (\i) -- (-37+\i,2.05);
\draw[gray!60] (\i) -- (-38+\i,2.05);
}

\foreach \i in {61,62,63}
{\path (-54+\i,4.5) coordinate (\i);
\node[above, gray!60] at (\i) {\i};
\draw[gray!60] (\i) -- (-52+\i,3.55);
\draw[gray!60] (\i) -- (-53+\i,3.55);
\draw[gray!60] (\i) -- (-54+\i,3.55);
\draw[gray!60] (\i) -- (-55+\i,3.55);
\draw[gray!60] (\i) -- (-56+\i,3.55);
}

\end{tikzpicture}
\caption{The Hasse diagram of the induced subposet $\tilde{P}_{\{8,9,10\}}$ of $P_{\{16,17,18,19,20\}}$}\label{fig:tilde8}
\end{figure}
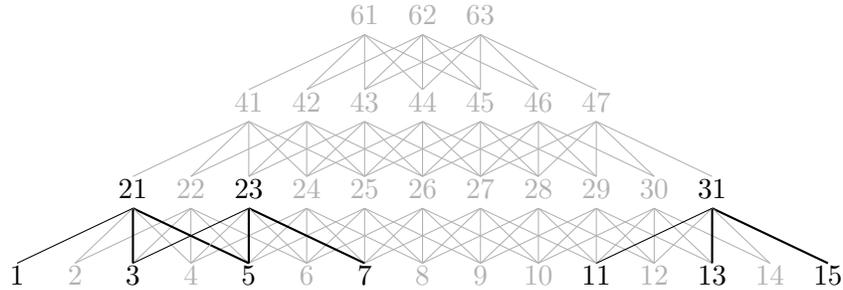

\begin{figure}[ht!]
\centering
\begin{tikzpicture}[scale=0.77]

\foreach \i in {1,3,5,7,13,15,17}
{\path (\i,0) coordinate (\i);
\node[above] at (\i) {\i};}

\foreach \i in {2,4,6,8,9,10,11,12,14,16}
{\path (\i,0) coordinate (\i);
\node[above, gray!60] at (\i) {\i};}

\foreach \i in {24,26,27,28,29,30,31,32,33,34}
{\path (-20+\i,1.5) coordinate (\i);
\node[above, gray!60] at (\i) {\i};
\draw[gray!60] (\i) -- (-18+\i,0.55);
\draw[gray!60] (\i) -- (-19+\i,0.55);
\draw[gray!60] (\i) -- (-20+\i,0.55);
\draw[gray!60] (\i) -- (-21+\i,0.55);
\draw[gray!60] (\i) -- (-22+\i,0.55);
}

\foreach \i in {23,25,35}
{\path (-20+\i,1.5) coordinate (\i);
\node[above] at (\i) {\i};
\draw[thick] (\i) -- (-18+\i,0.55);
\draw[gray!60] (\i) -- (-19+\i,0.55);
\draw[thick] (\i) -- (-20+\i,0.55);
\draw[gray!60] (\i) -- (-21+\i,0.55);
\draw (\i) -- (-22+\i,0.55);
}

\foreach \i in {45,46,47,48,49,50,51,52,53}
{\path (-40+\i,3) coordinate (\i);
\node[above, gray!60] at (\i) {\i};
\draw[gray!60] (\i) -- (-38+\i,2.05);
\draw[gray!60] (\i) -- (-39+\i,2.05);
\draw[gray!60] (\i) -- (-40+\i,2.05);
\draw[gray!60] (\i) -- (-41+\i,2.05);
\draw[gray!60] (\i) -- (-42+\i,2.05);
}

\foreach \i in {67,68,69,70,71}
{\path (-60+\i,4.5) coordinate (\i);
\node[above, gray!60] at (\i) {\i};
\draw[gray!60] (\i) -- (-58+\i,3.55);
\draw[gray!60] (\i) -- (-59+\i,3.55);
\draw[gray!60] (\i) -- (-60+\i,3.55);
\draw[gray!60] (\i) -- (-61+\i,3.55);
\draw[gray!60] (\i) -- (-62+\i,3.55);
}

\end{tikzpicture}
\caption{The Hasse diagram of the induced subposet $\tilde{P}_{\{9,10,11\}}$ of $P_{\{18,19,20,21,22\}}$}\label{fig:tilde9}
\end{figure}
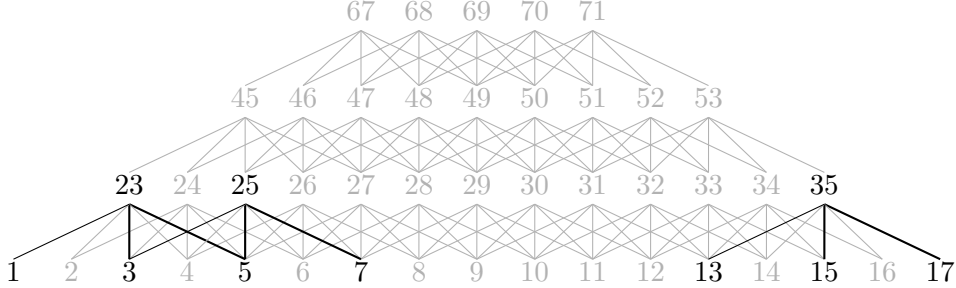

Now, we restate Proposition \ref{prop:FMS} by using the poset we constructed.

\begin{prop}\label{prop:lower}
Let $\la$ be a self-conjugate partition. Then $\la$ is an $(s,s+1,s+2)$-core partition if and only if the set $MD(\la)$ is a lower ideal of $\tilde{P}_{\{s,s+1,s+2\}}$ with no elements $h_1,h_2$ such that $h_1+h_2\in \{2s,2s+2,2s+4\}$.
\end{prop}

\begin{example}  For a self-conjugate $(8,9,10)$-core partition $\la=(6,3,3,1,1,1)$, the set $MD(\la)=\{11,3,1\}$ of main diagonal hook lengths is a lower ideal of $\tilde{P}_{\{8,9,10\}}$ with no elements $h_1,h_2$ such that $h_1+h_2\in\{16,18,20\}$.
\end{example} 

For convenience, we add dotted edges connecting elements $h_1$, $h_2$ in the Hasse diagram of $\tilde{P}_{\{s,s+1,s+2\}}$ with $h_1+h_2\in \{2s,2s+2,2s+4\}$ so that at most one end point of each dotted edge can be selected for the lower ideal corresponding to an $(s,s+1,s+2)$-core. From now on, we use the \emph{modified diagram} for $\tilde{P}_{\{s,s+1,s+2\}}$ as in Figure \ref{fig:modified8}.

\begin{figure}[ht!]
\centering
\begin{tikzpicture}[scale=0.6]

\foreach \i in {1,3,5,7}
{\path (\i,0) coordinate (\i);
\node[above] at (\i) {\i};}

\foreach \i in {11,13,15}
{\path (17-\i,-1.5) coordinate (\i);
\node[above] at (\i) {\i};}

\foreach \i in {21,23}
{\path (-18+\i,1.5) coordinate (\i);
\node[above] at (\i) {\i};
\draw (\i) -- (-16+\i,0.75);
\draw (\i) -- (-18+\i,0.75);
\draw (\i) -- (-20+\i,0.75);
}

\path (4,-2.25) coordinate (31);
\node[below] at (31) {31}; 
\draw (31)-- (15)
       (31) -- (13)
       (31) -- (11);
       
\foreach \i in {1,3,5}
\draw[dotted] (\i)-- (2,-0.75);

\foreach \i in {3,5,7}
\draw[dotted] (\i)-- (4,-0.75);

\foreach \i in {5,7}
\draw[dotted] (\i)-- (6,-0.75);

\end{tikzpicture}
\qquad \qquad \quad\begin{tikzpicture}[scale=0.6]

\foreach \i in {1,3,5,7}
{\path (\i,0) coordinate (\i);
\node[above] at (\i) {\i};}

\foreach \i in {13,15,17}
{\path (19-\i,-1.5) coordinate (\i);
\node[above] at (\i) {\i};}

\foreach \i in {23,25}
{\path (-20+\i,1.5) coordinate (\i);
\node[above] at (\i) {\i};
\draw (\i) -- (-18+\i,0.75);
\draw (\i) -- (-20+\i,0.75);
\draw (\i) -- (-22+\i,0.75);
}

\path (4,-2.25) coordinate (35);
\node[below] at (35) {35}; 
\draw (35)-- (15)
       (35) -- (13)
       (35) -- (17);
       
\foreach \i in {1,3,5}
\draw[dotted] (\i)-- (2,-0.75);

\foreach \i in {3,5,7}
\draw[dotted] (\i)-- (4,-0.75);

\foreach \i in {5,7}
\draw[dotted] (\i)-- (6,-0.75);

\end{tikzpicture}
\caption{The modified diagrams of $\tilde{P}_{\{8,9,10\}}$ and $\tilde{P}_{\{9,10,11\}}$}\label{fig:modified8}
\end{figure}
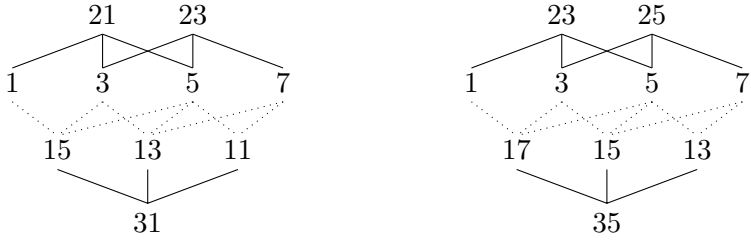

We note that
$$Q\cong P_{\{\lfloor \frac{s}{2}\rfloor+1,\lfloor \frac{s}{2}\rfloor+2,\lfloor \frac{s}{2}\rfloor+3 \}}\quad \text{and} \quad  R\cong P_{\{\lfloor \frac{s}{2}\rfloor,\lfloor \frac{s}{2}\rfloor+1,\lfloor \frac{s}{2}\rfloor+2 \}},$$
and therefore, $\tilde{P}_{\{2s,2s+1,2s+2\}}$ is equivalent to $\tilde{P}_{\{2s+1,2s+2,2s+3\}}$. Moreover, it is not hard to notice that the modified diagrams of $\tilde{P}_{\{2s,2s+1,2s+2\}}$ and $\tilde{P}_{\{2s+1,2s+2,2s+2\}}$ are also equivalent. Thus, we have the following proposition. 

\begin{prop}\label{prop:evenodd}
For a positive integer $s$, the number of self-conjugate $(2s,2s+1,2s+2)$-cores are equal to the number of self-conjugate $(2s+1,2s+2,2s+3)$-cores.
\end{prop}

\section{Counting self-conjugate simultaneous core partitions}\label{sec:counting}

In this section, we give a formula for the number of symmetric Motzkin paths, and then show that the number of self-conjugate $(2s,2s+1,2s+2)$-cores and the number of symmetric Motzkin paths of length $s$ satisfy the same recurrence relation.

\subsection{Counting symmetric Motzkin paths}

For a fixed $i$, there are $C_i \binom{n}{2i}$ Motzkin paths with exactly $i$ up steps since there are $C_i$ Dyck paths and there are $\binom{n}{2i}$ ways to insert $n-2i$ flat steps into a Dyck path with $i$ up steps. We say that a Motzkin path of length $n$ is \emph{symmetric} if its reflection about the line $x=\frac{n}{2}$ is itself. Let $S_n$ denote the number of symmetric Motzkin paths of length $n$. For exmaple, $S_0=1,S_1=1,S_2=2,S_3=2,S_4=5$. 

\begin{figure}[ht!]
\centering
\begin{tikzpicture}[scale=0.5]
\draw (0,0) -- (1,1) -- (2,0) -- (3,1) -- (4,0);
\draw[-|] (5+0,0) -- (5+1,1);
\draw[-|] (5+2,2) -- (5+3,1);
\draw (5+1,1) -- (5+2,2)
      (5+3,1) -- (5+4,0);
\draw (10+0,0) -- (10+1,0) -- (10+2,1) -- (10+3,0) -- (10+4,0);
\draw[-|] (15+0,0) -- (15+1,1) -- (15+2,1); 
\draw (15+2,1)-- (15+3,1) -- (15+4,0);
\draw[-|] (20+0,0)-- (20+1,0);
\draw[-|] (20+1,0)-- (20+2,0);
\draw[-|] (20+2,0)-- (20+3,0);
\draw (20+3,0)-- (20+4,0);
\end{tikzpicture}
\caption{Symmetric Motzkin paths of length $4$}
\end{figure}
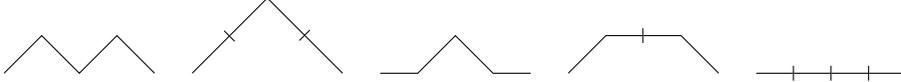

We note that the $(n+1)$st step of any symmetric Motzkin path of length $2n+1$ must be a flat step, and therefore, there is a natural bijection between symmetric Motzkin paths of length $2n+1$ and that of length $2n$ so that $S_{2n+1}=S_{2n}$.

Now, we count the number of symmetric Motzkin paths.

\begin{prop}\label{prop:symmetric} 
The number of symmetric Motzkin paths of length $n$ is
$$S_n=\sum_{i\geq0} \binom{\lfloor \frac{n}{2}\rfloor}{i}\binom{i}{\lfloor \frac{i}{2} \rfloor}.$$
\end{prop}

\begin{proof}
It is enough to enumerate symmetric Motzkin paths of length $2n$. Suppose we are given a symmetric Dyck path with $i$ up steps, so that it has $2n-2i$ flat steps. To obtain a symmetric Motzkin path of length $2n$ with $i$ up steps, it is enough to consider inserting $n-i$ flat steps into the first half of the given symmetric Dyck path. Since there are $\binom{i}{\lfloor i/2\rfloor}$ symmetric Dyck paths with $i$ up steps as in Remark \ref{rem:symmetric}, and there are $\binom{n}{i}$ ways to insert flat steps, the number of symmetric Motzkin paths of length $2n$ with $i$ up steps is $\binom{n}{i}\binom{i}{\lfloor i/2 \rfloor}$. Therefore, $S_{2n}=S_{2n+1}=\sum_{i\geq0}\binom{n}{i}\binom{i}{\lfloor i/2 \rfloor}$. 
\end{proof} 

Now, we consider a recurrence relation of $S_{2n}$ involving $M_{n}$. For a symmetric Motzkin path $P=P_1P_2\cdots P_{2n}$ of length $2n$, where $P_i$ denotes the $i$th step, let $k\leq n$ be the largest number such that $P$ meets $x$-axis at $(k,0)$.
We note that if $k=n$, then both of $P_1P_2\cdots P_{n}$ and $P_{n+1}P_{n+2}\cdots P_{2n}$ are Motzkin paths of length $n$ which are symmetric to each other. 
On the other hand, if $k<n$, then $P_{k+1}=U$, $P_{2n-k}=D$, the subpath $P_{k+2}P_{k+3}\cdots P_{2n-k-1}$ is a symmetric Motzkin path of length $2n-2k-2$, and both of two subpaths $P_1P_2\cdots P_k$ and $P_{2n-k+1}P_{2n-k+2}\cdots P_{2n}$ are Motzkin paths of length $k$ which are symmetric to each other. Hence, we have a relation between $S_{2n}$ and $M_n$:
\begin{equation} \label{eqn:re}
S_{2n}=M_n+\sum_{k=0}^{n-1}S_{2n-2k-2}M_{k}.              
\end{equation}   

Equation (\ref{eqn:re}) and a closed formula for $S_{2n}$ can also be found in the OEIS as A005773 \cite{OEIS}.

\subsection{Counting self-conjugate $(2s,2s+1,2s+2)$-cores}  

The following lemma plays an important role to obtain a recurrence relation for the number of self-conjugate $(2s,2s+1,2s+2)$-cores. 

\begin{lem} \label{lem:with2s-1}
Let $s$ be a positive integer.
The number of self-conjugate $(2s,2s+1,2s+2)$-cores $\la$ with $2s-1\in MD(\la)$ is equal to the number of self-conjugate $(2s-2,2s-1,2s)$-cores.
\end{lem}

\begin{proof} 
By Proposition \ref{prop:lower}, there is a bijection between $(2s,2s+1,2s+2)$-cores $\la$ with $2s-1\in MD(\la)$ and lower ideals $I$ of $\tilde{P}_{\{2s,2s+1,2s+2\}}$ containing $2s-1$ and no elements $h_1,h_2$ such that $h_1+h_2\in\{4s,4s+2,4s+4\}$. Thus, it is enough to consider lower ideals of the first diagram in Figure \ref{fig:modified2s}.

To prove the lemma, we construct a bijection $\phi$ between lower ideals $I$ of the first diagram and lower ideals $J$ of the second diagram in Figure \ref{fig:modified2s}, since the second diagram is equivalent to the modified diagram of $\tilde{P}_{\{2s-2,2s-1,2s\}}$.

\begin{figure}[ht!]
\centering
\begin{tikzpicture}[scale=0.65]

\foreach \i in {1,3,5,7,9,11}
{\path (\i,0) coordinate (\i);}
\node[above] at (1) {$1$};
\node[above] at (3) {$3$};
\node[above] at (5) {$\cdots$};
\node[above] at (7) {$\cdots$};
\node[above] at (9) {$2s-3$};
\node[above] at (11) {$\mathbf{2s-1}$};
 
\path (2,1.5) coordinate (2);
\node[above] at (2) {$4s+5$};
\draw (2) -- (1,0.7);
\draw (2) -- (3,0.7);
\draw (2) -- (5,0.7);

\path (4,1.5) coordinate (4);
\node[above] at (4) {$\cdots$};
\draw (4) -- (3,0.7);
\draw (4) -- (5,0.7);
\draw (4) -- (7,0.7);

\path (6,1.5) coordinate (6);
\node[above] at (6) {$\cdots$};
\draw (6) -- (5,0.7);
\draw (6) -- (7,0.7);
\draw (6) -- (9,0.7);

\path (8,1.5) coordinate (6);
\node[above] at (6) {$6s-1$};
\draw (6) -- (7,0.7);
\draw (6) -- (9,0.7);

\path (4,3) coordinate (8);
\node[above] at (8) {$\cdots$};
\draw (8) -- (2,2.2);
\draw (8) -- (4,2.2);
\draw (8) -- (6,2.2);

\path (6,3) coordinate (10);
\node[above] at (10) {$\cdots$};
\draw (10) -- (4,2.2);
\draw (10) -- (6,2.2);
\draw (10) -- (8,2.2);

\foreach \i in {13,15,17,19,21}
{\path (24-\i,-1.5) coordinate (\i);}

\node[above] at (21) {$4s-1$};
\node[above] at (19) {$\cdots$};
\node[above] at (17) {$2s+7$};
\node[above,gray!60] at (15) {$2s+5$};
\node[above,gray!60] at (13) {$2s+3$};

\path (5,-2.25) coordinate (31);
\node[below] at (31) {$\cdots$}; 
\draw (31)-- (21)
       (17) -- (31) -- (19);

\foreach \i in {1,3,5}
\draw[dotted] (\i)-- (3,-0.9);

\foreach \i in {3,5,7}
\draw[dotted] (\i)-- (5,-0.9);

\foreach \i in {5,7,9}
\draw[dotted] (\i)-- (7,-0.9);

\foreach \i in {11}
\draw[dotted] (\i)-- (9,-0.9);

\foreach \i in {11}
\draw[dotted] (\i)-- (11,-0.9);

\end{tikzpicture} 
\quad
\begin{tikzpicture}[scale=0.65]

\node[above] at (-0.4,0) {$\cong$};
\foreach \i in {1,3,5,7,9}
{\path (\i,0) coordinate (\i);}
\node[above] at (1) {$1$};
\node[above] at (3) {$3$};
\node[above] at (5) {$\cdots$};
\node[above] at (7) {$\cdots$};
\node[above] at (9) {$2s-3$};

\path (2,1.5) coordinate (2);
\node[above] at (2) {$4s+5$};
\draw[dotted] (2) -- (1,0.7);
\draw[dotted] (2) -- (3,0.7);
\draw[dotted] (2) -- (5,0.7);

\path (4,1.5) coordinate (4);
\node[above] at (4) {$\cdots$};
\draw[dotted] (4) -- (3,0.7);
\draw[dotted] (4) -- (5,0.7);
\draw[dotted] (4) -- (7,0.7);

\path (6,1.5) coordinate (6);
\node[above] at (6) {$\cdots$};
\draw[dotted] (6) -- (5,0.7);
\draw[dotted] (6) -- (7,0.7);
\draw[dotted] (6) -- (9,0.7);

\path (8,1.5) coordinate (6);
\node[above] at (6) {$6s-1$};
\draw[dotted] (6) -- (7,0.7);
\draw[dotted] (6) -- (9,0.7);

\path (4,3) coordinate (8);
\node[above] at (8) {$\cdots$};
\draw (8) -- (2,2.2);
\draw (8) -- (4,2.2);
\draw (8) -- (6,2.2);

\path (6,3) coordinate (10);
\node[above] at (10) {$\cdots$};
\draw (10) -- (4,2.2);
\draw (10) -- (6,2.2);
\draw (10) -- (8,2.2);

\foreach \i in {17,19,21}
{\path (24-\i,-1.5) coordinate (\i);}

\node[above] at (21) {$4s-13$};
\node[above] at (19) {$\cdots$};
\node[above] at (17) {$2s+7$};

\path (5,-2.25) coordinate (31);
\node[below] at (31) {$\cdots$}; 
\draw (31)-- (21)
       (17) -- (31) -- (19);

\foreach \i in {1,3,5}
\draw (\i)-- (3,-0.9);

\foreach \i in {3,5,7}
\draw (\i)-- (5,-0.9);

\foreach \i in {5,7,9}
\draw (\i)-- (7,-0.9);
\end{tikzpicture}

\caption{The modified diagram of $\tilde{P}_{\{2s,2s+1,2s+2\}}$ for ideals having  $2s-1$}\label{fig:modified2s}
\end{figure}
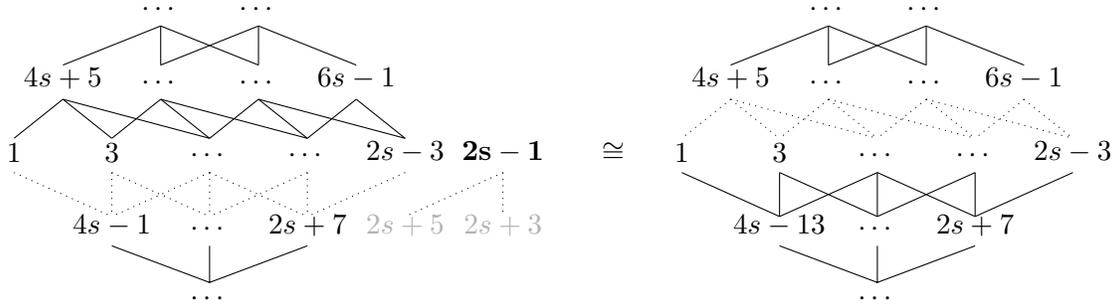

If $I$ is a lower ideal of the first diagram, then $I$ satisfies the following: 
\begin{itemize}
\item $I$ contains $2s-1$.
\item For $4s+5\leq h\leq 6s-1$, $h\in I$ implies $h-4s-4, h-4s-2, h-4s\in I$.
\item For $2s+7\leq h\leq 4s-1$, $h\in I$ implies $4s-h, 4s-h+2, 4s-h+4\not\in I$.
\end{itemize}
 
Now, we construct a corresponding set $\phi(I)$ from $I$ as follows.
 
\begin{itemize}
\item For each $h\in I$ with $4s+5\leq h\leq 6s-1$, delete $h-4s-4, h-4s-2, h-4s$ from $I$.
\item For each $h\in I$ with $2s+7\leq h\leq 4s-1$, add $4s-h, 4s-h+2, 4s-h+4$ to $I$.
\item Delete $2s-1$ from the set.
\end{itemize}

Then $\phi(I)$ is a lower ideal of the poset structure defined by the second diagram in Figure \ref{fig:modified2s} and it is easy to check that $\phi$ is a bijection.
\end{proof}

\begin{example} For self-conjugate $(8,9,10)$-cores $\la$ such that $7\in MD(\la)$, let $I_1, I_2,\dots,I_{13}$ be their corresponding lower ideals of $\tilde{P}_{\{8,9,10\}}$.  Now, we list $I_i$ and $J_i=\phi(I_i)$ for $i=1,2,\dots,13$, where $\phi$ is the bijection defined in the proof of Lemma \ref{lem:with2s-1}.\\

$\begin{array}{lr}
I_1=\{7\}\\
I_3=\{3,7\}\\
I_5=\{1,3,7\}\\
I_7=\{3,5,7\}\\
I_9=\{1,3,5,7,21\}\\
I_{11}=\{3,5,7,23\} \\
I_{13}=\{7,15\} 
\end{array}
$
$\begin{array}{lr}
J_1=\emptyset\\
J_3=\{3\}\\
J_5=\{1,3\} \\
J_7=\{3,5\}\\
J_9=\{21\}\\
J_{11}=\{23\} \\
J_{13}=\{1,3,5,15\} 
\end{array}
$
\quad\qquad
$\begin{array}{lr}
I_2=\{1,7\}\\
I_4=\{5,7\}\\
I_6=\{1,5,7\}\\
I_8=\{1,3,5,7\}\\
I_{10}=\{1,3,5,7,21,13\}\\
I_{12}=\{1,3,5,7,23\}
\end{array}
$
$\begin{array}{lr}
J_2=\{1\}\\
J_4=\{5\}\\
J_6=\{1,5\}\\
J_8=\{1,3,5\}\\
J_{10}=\{21,23\}\\
J_{12}=\{1,23\}
\end{array}
$\\

We note that for each $i$, $I_i$ is an ideal of the first diagram and $J_i$ is an ideal of the second diagram of the following figure.
 
\begin{figure}[ht!]
\centering
\begin{tikzpicture}[scale=0.65]

\foreach \i in {5,7,9,11}
{\path (\i,0) coordinate (\i);}
\node[above] at (5) {$1$};
\node[above] at (7) {$3$};
\node[above] at (9) {$5$};
\node[above] at (11) {$\mathbf{7}$};

\path (6,1.5) coordinate (6);
\node[above] at (6) {$21$};
\draw (6) -- (5,0.7);
\draw (6) -- (7,0.7);
\draw (6) -- (9,0.7);

\path (8,1.5) coordinate (6);
\node[above] at (6) {$23$};
\draw (6) -- (7,0.7);
\draw (6) -- (9,0.7);

\foreach \i in {13,15,17}
{\path (24-\i,-1.5) coordinate (\i);}

\node[above] at (17) {$15$};
\node[above,gray!60] at (15) {$13$};
\node[above,gray!60] at (13) {$11$};

\foreach \i in {5,7,9}
\draw[dotted] (\i)-- (7,-0.8);

\foreach \i in {11}
\draw[dotted] (\i)-- (9,-0.8);

\foreach \i in {11}
\draw[dotted] (\i)-- (11,-0.8);

\end{tikzpicture} 
\quad
\begin{tikzpicture}[scale=0.65]

\node[above] at (3.6,0) {$\cong$};
\foreach \i in {5,7,9}
{\path (\i,0) coordinate (\i);}
\node[above] at (5) {$1$};
\node[above] at (7) {$3$};
\node[above] at (9) {$5$};

\path (6,1.5) coordinate (6);
\node[above] at (6) {$21$};
\draw[dotted] (6) -- (5,0.7);
\draw[dotted] (6) -- (7,0.7);
\draw[dotted] (6) -- (9,0.7);

\path (8,1.5) coordinate (6);
\node[above] at (6) {$23$};
\draw[dotted] (6) -- (7,0.7);
\draw[dotted] (6) -- (9,0.7);

\foreach \i in {17}
{\path (24-\i,-1.5) coordinate (\i);}
\node[above] at (17) {$15$};

\foreach \i in {5,7,9}
\draw (\i)-- (7,-0.8);
\end{tikzpicture}
\caption{The modified diagram of $\tilde{P}_{\{8,9,10\}}$ for ideals $I$ with $7\in I$}\label{fig:exam}
\end{figure}
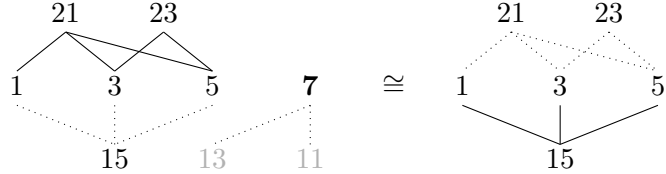

\end{example}

The following proposition is a generalization of Lemma \ref{lem:with2s-1}.

\begin{prop} \label{prop:no2s-1}
Let $s$ and $k$ be positive integers such that $k\leq s$.
\begin{enumerate}
\item[(a)] The number of self-conjugate $(2s,2s+1,2s+2)$-cores $\la$ satisfies that 
$$2k-1\in MD(\la) \quad \text{and} \quad 2k+1,2k+3,\dots, 2s-1\not\in MD(\la)$$ 
is the number of self-conjugate $(2k-2,2k-1,2k)$-cores multiplied by $M_{s-k}$.
\item[(b)] The number of self-conjugate $(2s,2s+1,2s+2)$-cores $\la$ with 
$1,3,\dots, 2s-1\not\in MD(\la)$ is $M_{s}$.
\end{enumerate}
\end{prop} 
\begin{proof}
We consider the modified diagram of $\tilde{P}_{\{2s,2s+1,2s+2\}}$ with restrictions in Figure \ref{fig:modified2s2}.

\begin{figure}[ht!]
\centering
\begin{tikzpicture}[scale=0.72]

\foreach \i in {1,3,5,7,9,11,13,15,17,19,21}
{\path (\i,0) coordinate (\i);}
\node[above] at (1) {\small$1$};
\node[above] at (3) {\small$3$};
\node[above] at (5) {\small$\cdots$};
\node[above] at (7) {\small$\cdots$};
\node[above] at (9) {\small$2k-3$};
\node[above] at (11) {\small$\mathbf{2k-1}$};
\node[above,gray!60] at (13) {\small$2k+1$};
\node[above,gray!60] at (15) {\small$2k+3$};
\node[above,gray!60] at (17) {\small$\cdots$};
\node[above,gray!60] at (19) {\small$2s-3$};
\node[above,gray!60] at (21) {\small$2s-1$};
  
\path (2,1.5) coordinate (2);
\node[above] at (2) {\small$4s+5$};
\draw (2) -- (1,0.7);
\draw (2) -- (3,0.7);
\draw (2) -- (5,0.7);

\path (4,1.5) coordinate (4);
\node[above] at (4) {\small$\cdots$};
\draw (4) -- (3,0.7);
\draw (4) -- (5,0.7);
\draw (4) -- (7,0.7);

\path (6,1.5) coordinate (6);
\node[above] at (6) {\small$\cdots$};
\draw (6) -- (5,0.7);
\draw (6) -- (7,0.7);
\draw (6) -- (9,0.7);

\path (8,1.5) coordinate (6);
\node[above] at (6) {\small$4s+2k-1$};
\draw (6) -- (7,0.7);
\draw (6) -- (9,0.7);

\path (4,3) coordinate (8);
\node[above] at (8) {\small$\cdots$};
\draw (8) -- (2,2.2);
\draw (8) -- (4,2.2);
\draw (8) -- (6,2.2);

\path (6,3) coordinate (10);
\node[above] at (10) {\small$\cdots$};
\draw (10) -- (4,2.2);
\draw (10) -- (6,2.2);
\draw (10) -- (8,2.2);

\foreach \i in {111,113,115,117,119,121,123,125,127}
{\path (130-\i,-1.5) coordinate (\i);}

\node[above] at (127) {\small$4s-1$};
\node[above] at (125) {\small$\cdots$};
\node[above] at (123) {\small$4s-2k+7$};
\node[above] at (115) {\small$4s-2k-1$};
\node[above] at (113) {\small$\cdots$};
\node[above] at (111) {\small$2s+3$};

\path (5,-2.25) coordinate (31);
\node[below] at (31) {\small$\cdots$}; 
\draw (31)-- (127)
       (123) -- (31) -- (125);

\path (17,-2.25) coordinate (33);
\node[below] at (33) {\small$\cdots$}; 
\draw (33)-- (111)
       (113) -- (33) -- (115);       

\foreach \i in {1,3,5}
\draw[dotted] (\i)-- (3,-0.9);

\foreach \i in {3,5,7}
\draw[dotted] (\i)-- (5,-0.9);

\foreach \i in {5,7,9}
\draw[dotted] (\i)-- (7,-0.9);

\foreach \i in {9,11,13}
\draw[dotted] (11)-- (\i,-0.9);

\foreach \i in {13,15,17}
\draw[dotted] (\i)-- (15,-0.9);

\foreach \i in {15,17,19}
\draw[dotted] (\i)-- (17,-0.9);

\foreach \i in {17,19,21}
\draw[dotted] (\i)-- (19,-0.9);
\end{tikzpicture}

\caption{The modified diagram of $\tilde{P}_{\{2s,2s+1,2s+2\}}$ with restrictions}\label{fig:modified2s2}
\end{figure}
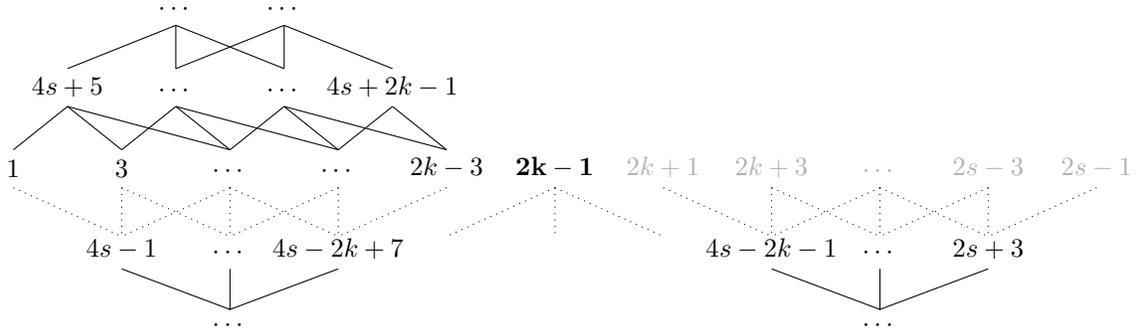

\begin{enumerate}
\item[(a)] On the modified diagram with restrictions, the left-hand side of $2k-1$ is equivalent to the modified diagram of $\tilde{P}_{\{2k-2,2k-1,2k\}}$ as we showed in Lemma \ref{lem:with2s-1}, and the right-hand side of $2k-1$ is equivalent to the Hasse diagram of the poset $P_{\{2s-2k-2,2s-2k-1,2s-2k\}}$. Hence, by Corollary \ref{cor:Motzkin}, the number of lower ideals $I$ satisfying $2k-1\in I$ and $2k+1,2k+3,\dots,2s-1\not \in I$ is $M_{s-k}$ times the number of self-conjugate $(2k-2,2k-1,2k)$-cores.
\item[(b)] If $1,3,\dots, 2s-1\not\in I$, then $I$ is a lower ideal of a subposet of $\tilde{P}_{\{2s,2s+1,2s+2\}}$ which is equivalent to $P_{\{s,s+1,s+2\}}$. Hence, by Corollary \ref{cor:Motzkin}, the number of lower ideals $I$ with $1,3,\dots,2s-1\not \in I$ is $M_{s}$. 
\end{enumerate}
\end{proof}

Now, we are ready to prove our main theorem.

\begin{thm}
For a positive integer $s$, the number of self-conjugate $(s,s+1,s+2)$-cores is
$$\sum_{i\geq0}\binom{\lfloor \frac{s}{2}\rfloor}{i}\binom{i}{\lfloor \frac{i}{2} \rfloor},$$
which counts the number of symmetric Motzkin paths of length $s$.
\end{thm}

\begin{proof}
 Let $a_s$ denote the number of self-conjugate $(s,s+1,s+2)$-cores. From Proposition \ref{prop:evenodd}, we have $a_{2s+1}=a_{2s}$. Hence, it is enough to show that $a_{2s}=S_{2s}$. 
 
For $1\leq k \leq s$, let $A_k$ be the set of self-conjugate $(2s,2s+1,2s+2)$-cores $\la$ that satisfies
$$2k-1\in MD(\la) \quad \text{and} \quad 2k+1,2k+3,\dots, 2s-1\not\in MD(\la),$$ and let $A_0$ be the set of self-conjugate $(2s,2s+1,2s+2)$-cores $\la$ with $2i-1\not \in MD(\la)$ for $1\leq i\leq s$. Then, $A_0\cup A_1\cup \cdots \cup A_s$ is the set of self-conjugate $(2s,2s+1,2s+1)$-cores and 
$$a_{2s}=|A_0|+|A_1|+\cdots+|A_s|.$$
From Proposition \ref{prop:no2s-1}, we have $|A_0|=M_s$ and $|A_k|=a_{2k-2}M_{s-k}$ for $1\leq k \leq s$, and therefore, 
$$a_{2s}=M_s+\sum_{k=1}^{s}a_{2k-2}M_{s-k}=M_s+\sum_{k=0}^{s-1}a_{2s-2k-2}M_{k}.$$
Since the relation between $a_{2s}$ and $M_{s}$ is equivalent to (\ref{eqn:re}) and
$a_0=S_0=1$, we have come to a conclusion that $a_{2s}=S_{2s}=\sum\binom{s}{i}\binom{i}{\lfloor i/2 \rfloor}$ by Proposition \ref{prop:symmetric}.
\end{proof}

Encouraged by this success, we offer the following generalized conjecture.

\begin{conj} For given positive integers $s$ and $k$, the number of self-conjugate $(s,s+1,\dots,s+k)$-cores is equal to the number of symmetric $(s,k)$-generalized Dyck paths.
\end{conj}

\bibliographystyle{plain} 
\bibliography{mybib}

\end{document}